\documentclass[11pt]{amsart}
\usepackage{amsmath}
\usepackage{amsthm}
\usepackage{amssymb}
\usepackage{a4wide}
\usepackage{color}

\usepackage[pagewise,mathlines]{lineno}
\usepackage{hyperref}
\usepackage{amsfonts}

\newcommand*{\Bscr}{\mathcal B}

\newcommand*{\Fscr}{\mathcal F}
\newcommand*{\Gscr}{\mathcal G}

\newcommand*{\Tscr}{\mathcal T}
\newcommand*{\Uscr}{\mathcal U}
\newcommand*{\rs}[1]{{\vert}_{#1}} 

\newcommand*{\N}{\mathbb N}
\newcommand*{\R}{\mathbb R}
\newcommand*{\E}{\mathbb E}
\newcommand*{\Pp}{\mathbb P}

\newcommand{\cb}{{\mathcal B}}  
  
\newcommand{\cl}{{\mathcal L}}  
\newcommand{\cu}{{\mathcal U}}  
\newcommand{\bff}{{\boldsymbol f}}

\newcommand{\sgn}{\,{\rm sgn}}
\newtheorem{lemma}{Lemma}[section]
\newtheorem{cor}{Corollary}[section]
\newtheorem{prop}{Proposition}[section]
\newtheorem{theorem}{Theorem}[section]

\newtheorem{remark}{Remark}[section]

\def\rr{\mathbb{R}}

\def\eps{\varepsilon}

\def\supu{\overline {u}}
\def\subu{\underline{u}}

\def\tu{\widetilde U}
\def\tv{\widetilde V}
\def\tw{\widetilde W}
\def\ov{\overline v}

\begin{document}

\title[The flow  of a branching process]
{From Gaussian estimates for nonlinear evolution equations to the long time 
 behavior of  branching processes}

\author[L. Beznea, L. I. Ignat,
J. D. Rossi ]{Lucian Beznea, Liviu I. Ignat, and
Julio D. Rossi}

\address{L. Beznea
\hfill\break\indent  Simion Stoilow Institute of Mathematics  of  the Romanian Academy, 
\hfill\break\indent P.O. Box 1-764, RO-014700, Bucharest,  Romania, and
\hfill\break\indent University of Bucharest, Faculty of Mathematics and Computer Science }
\email{{\tt lucian.beznea@imar.ro}}

\address{L. I. Ignat
\hfill\break\indent  Simion Stoilow Institute of Mathematics  of  the Romanian Academy, 
\hfill\break\indent P.O. Box 1-764, RO-014700, Bucharest,  Romania, and
\hfill\break\indent University of Bucharest, Faculty of Mathematics and Computer Science }
 \email{{\tt
liviu.ignat@gmail.com}\hfill\break\indent  {\it Web page: }{\tt
http://www.imar.ro/\~\,lignat}}

\address{J. D. Rossi
\hfill\break\indent Dpto. de Matem{\'a}ticas, FCEyN,
Universidad de Buenos Aires, \hfill\break\indent 1428, Buenos Aires,
Argentina. } \email{{\tt jrossi@dm.uba.ar} \hfill\break\indent {\it
Web page: }{\tt http://mate.dm.uba.ar/$\sim$jrossi/}}

\begin{abstract} 
We study solutions to the evolution equation
 $u_t=\Delta u-u +\sum _{k\geqslant 1}q_ku^k$, $t>0$, in $\R^d$. Here the coefficients $q_k\geqslant 0$ verify
 $
\sum_{k\geqslant 1}q_k=1<
\sum_{k\geqslant 1}kq_k<\infty$.
  First, we deal with existence, uniqueness,  and the asymptotic
 behavior of the solutions as $t\to +\infty$. 
 We then deduce results on the long time behavior  of the  associated   branching process, 
 with state space the set of all finite configurations of $\R^d$, under the assumption that $\sum_{k\geq 1} k^2q_k<\infty$. 
It turns out that the distribution  of the branching process  behaves when  the time tends to infinity
like that of the Brownian motion on the set of all finite configurations of $\R^d$.
However, due to the lack of conservation of the total mass of the initial non linear equation, a deformation with a multiplicative coefficient occurs.
Finally, we establish  asymptotic properties of the occupation time of this branching process.
\end{abstract}

\maketitle



\section{Introduction}
\label{Sect.intro}
\setcounter{equation}{0}

In this paper we consider the following evolution problem
\begin{equation}\label{sys.u}
\left\{
\begin{array}{ll}
\displaystyle u_t (x,t)=\Delta u (x,t) -u  (x,t)+\sum _{k\geqslant 1}q_ku^k (x,t), \qquad & x\in \rr^d, \, t>0,\\[10pt]
0\leqslant u (x,t)\leqslant 1,\\[10pt]
u(x,0)=\varphi (x),
\end{array}
\right.
\end{equation}
and analyze the long time behavior of the solutions.
We assume that $q_k$, $k\geqslant 1$, are nonnegative numbers satisfying
\begin{equation}\label{hip.1.q}
\sum_{k\geqslant 1}q_k=1,
\end{equation}
and
\begin{equation}\label{hip.2.q}
1<\sum_{k\geqslant 1}kq_k=q<\infty.
\end{equation}

The main results obtained here regarding equation \eqref{sys.u} are stated below. 

\begin{theorem}\label{existence}  \label{th1.1}
For any initial datum $\varphi\in L^\infty(\rr^d)$ such that $0\leqslant \varphi\leqslant 1$ there exists a unique global mild solution
$u\in C([0,\infty), L^\infty(\rr^d))$ to \eqref{sys.u} that satisfies
$0\leqslant u(x,t)\leqslant 1$ for every $x\in \rr^d$, $t\geqslant 0$, 
and
\begin{equation}\label{prop.3}
\|u(\cdot, t)\|_{L^\infty(\rr^d)}\leqslant \|\varphi\|_{L^\infty(\rr^d)}.
\end{equation}
If in addition $\varphi\in L^1(\rr^d)$ then $u\in C([0,\infty), L^1(\rr^d)\cap L^\infty(\rr^d))$, 
and
\begin{equation}\label{prop.2}
\int_{\rr^d} u(x,t)\, dx\leqslant \int _{\rr^d}\varphi (x) \, dx.
\end{equation}
Moreover, for any $\varphi_1,\varphi_2\in L^1(\rr^d)$ with $0\leqslant \varphi_1,\varphi_2\leqslant 1$,  
the corresponding solutions $u_1$ and $u_2$ satisfy
\begin{equation}
\label{prop.4}
  \|u_1-u_2\|_{L^\infty ([0,T],L^1(\rr^d))}\leqslant C(T, q) \|\varphi_1-\varphi_2\|_{L^1(\rr^d)}.
\end{equation}
\end{theorem}

We emphasize that for initial data $\varphi\in L^{1}(\rr^{d})\cap L^{\infty}(\rr^{d})$, classical regularity results for the linear heat equation and a bootstrap argument guarantee that in fact the mild solution obtained above satisfies  
$u\in C((0,\infty), W^{{2,p}}(\rr^{d}))\cap C^{1}((0,\infty), L^{p}(\rr^{d}))$ for all $1\leq p\leq 
 \infty$.

Once the well-possedness of system \eqref{sys.u} is established, 
we analyze the asymptotic behavior of the solutions.  
Assumption \eqref{hip.2.q} implies   that $q_1<1$,  otherwise the system is trivially reduced to a linear problem. 
We use $p$ for the index of the fist nonzero term, $q_p$, of  $(q_k)_{k\geq 2}$.

\begin{theorem}\label{asymptotic}
For any $\varphi\in L^1(\rr^d)$ with $0\leqslant \varphi\leqslant 1$ the solution to \eqref{sys.u} satisfies
\begin{equation}\label{est.1}
e^{-(1-q_1)t}(K_t\ast \varphi)(x)\leqslant u(x,t)\leqslant  C(d,p, \|\varphi\|_{L^1(\rr^d)})e^{-(1-q_1)t}(K_t\ast \varphi)(x),\quad x\in \rr^d,
\end{equation}
 where $K_t$ is the heat kernel. Moreover, there exists a positive constant $C_\varphi$, the total asymptotic mass,  such that
\begin{equation}\label{est.2}
t^{d/2}\|e^{t(1-q_1)}u(\cdot, t)- C_\varphi K_t \|_{L^\infty(\rr^d)} \rightarrow 0,\quad \text{as}\ t\rightarrow \infty.
\end{equation}
\end{theorem}

Now we move to a probabilistic interpretation of our results.
Recall that  equation \eqref{sys.u} is related to the  discrete branching Markov processes with state space
the set $\widehat{\R^d}$  of all positive measures on $\rr^d$  which are finite sums
of Dirac measures, called the {\it space of finite configurations} of $\rr^d$, with the following probabilistic interpretation.
An initial particle starts at a point of $\rr^d$  and moves according to the $d$-dimensional Brownian motion, 
until a random time when it is destroyed and replaced by a finite number of new particles, its direct descendants. 
The number $q_k$  is the probability that a particle destroyed  has precisely $k$ descendants.
Each direct descendant  starts at the terminal position of the parent particle and  
moves again according to the $d$-dimensional Brownian motion 
until its own terminal time when it is also destroyed and replaced by a new generation of particles and the process continues in this manner.

Recall that the $d$-dimensional Brownian motion induces in a canonical way a diffusion Markov process 
with state space the set  of finite configurations,  
describing  the movement of the systems of particles from $\widehat{\R^d}$  
without  having any branching: a system of $k$ particles moves according to $k$ independent Brownian motions;
we   call it {\it Brownian motion on $\widehat{\R^d}$}.
The probabilistic interpretation of our results is the following: 
the long time asymptotic  behavior of the distribution of the branching process on the space of finite configurations of $\R^d$,   
associated with  the sequence
$(q_k)_{k\geqslant 1}$ and  having as base process the $d$-dimensional Brownian motion, is the distribution of 
the $(1-q_1)$-subprocess of the Brownian motion on $\widehat{\R^d}$,
multiplied with the corresponding asymptotic mass.
As far as the authors know, in the case of equation \eqref{sys.u} there is no explicit dependence of the asymptotic mass $C_\varphi$ in terms of $\varphi$. In the case of the linear heat equation, i.e. $q_1=1$, the dependence is explicit $C_\varphi=\int _{\rr^d}\varphi$. 

Finally, we consider the issue of occupation times, that is, the times spent by the process in given subsets of the state space during a finite interval of time. 
Let us denote by $\widehat{X}$ the branching  processes induced by the solution $u(x,t)$ obtained in Theorem \ref{existence}. 
We consider weighted occupation times with respect functions
$\varphi \in b\mathcal{B}(\mathbb{R}^d)$.
Following \cite{Iscoe86}, we define the {\it weighted occupation time process} 
$Y_t(\varphi) : \Omega \to \overline{\mathbb{R}}$ as 
$$
Y_t(\varphi) : = \int_0^t\langle \varphi, \widehat X_s\rangle \, ds, \quad t \geqslant  0,
$$
where $\Omega$ is the path space of $\widehat{X}$. 
This process is a.s. real-valued. 

Here, we provide some estimates on this occupation time process. 
In fact, assuming $\sum _{k\geq 1} k^2 q_k<\infty$, 
for any $\varphi\in C_b(\rr^d)$, $\varphi\geq 0$ 
the weighted occupation time process $Y_t$ satisfies
\[
\lim _{t\rightarrow\infty}\E^\mu [\exp(-Y_t(\varphi))] =0, 
\]
and for $\alpha<(\sum _{k\geq 1}(k-1)q_k)^{-1}$
\[
\lim _{T\rightarrow \infty} \E^\mu  \Big[ \exp (- \frac {Y_T(\varphi)}{\exp(T/\alpha)})\Big]=1.
\]
This asymptotic behavior  of the occupation time of the branching process  is based on an integral representation of it,
a version of a result from \cite{Iscoe86} for measure-valued superprocesses.

The paper is organized as follows: in Section \ref{sect-behaviour} we deal with the PDE and prove
Theorem \ref{existence} and Theorem \ref{asymptotic}; 
in Section \ref{sect.3} we introduce the branching processes on the finite configurations of $\rr^d$ and analyze the long-time behavior. 
Related references involving branching processes and the associated nonlinear PDEs, are given in Remark \ref{rem3.1}.
Section \ref{sect.4} analyzes the weighted occupation time process induced by the branching process presented in Section \ref{sect.3}. 
Finally, we include an Appendix where we collect basic facts on the right Markov processes and some technical results used in Section \ref{sect.4}. 

\section{Well-possedness and asymptotic behavior for \eqref{sys.u}} \label{sect-behaviour}

Let us start showing the well-possedness of our problem.

\begin{proof}[Proof of Theorem \ref{existence}] Existence and uniqueness of global solutions 
is not difficult to prove. We include some details here for the sake of completeness.
We consider the subset $X_T$ of $C([0,T], L^\infty(\rr^d))$ given by
$$X_T=\{C([0,T], L^\infty(\rr^d)), 0\leqslant u\leqslant 1\}$$
endowed with the following norm
$$\|u\|_{X_T}=\max_{t\in [0,T]}\|u(t)\|_{L^\infty(\rr^d)}.$$

We consider the semigroup $S(t)$ associated with the following linear problem
\begin{equation}\label{linear}
\left\{
\begin{array}{ll}
\displaystyle u_t (x,t)=\Delta u (x,t) -u(x,t) , \qquad & x\in \rr^d,\, t>0,\\[10pt]
u(x,0)=\varphi (x).
\end{array}
\right.
\end{equation}
It follows that
$$S(t)\varphi=e^{-t}(K_t\ast \varphi)$$
where $K_t$ is the heat kernel. Observe that for any $p\in [1,\infty]$ we have
\begin{equation}\label{s.1}
\|S(t)\varphi\|_{L^p(\rr^d)}= e^{-t}\|K_t\ast \varphi \|_{L^p(\rr^d)} \leqslant e^{-t}\|\varphi\|_{L^p(\rr^d)}.
\end{equation}

Define $\Phi_u(t)$ as follows:
$$\Phi_u(t)= S(t)\varphi +\int _0^t S(t-s) \Big( \sum _{k\geqslant 1} q_ku^k(s)\Big)ds.$$
It is easy to see that if $0\leq \varphi\leq 1$ and $u\in X_T$ then $\Phi_u(t)\geqslant 0$. Also,
\begin{align*}
\|\Phi _u(t)\|_{L^\infty(\rr^d)}&\leqslant e^{-t}\|\varphi\|_{L^\infty(\rr^d)}+\int_0^t e^{-(t-s)}
\Big\| \sum _{k\geqslant 1} q_ku^k (s)\Big\|_{L^\infty(\rr^d)}ds\\
&\leqslant e^{-t}+\sum_{k\geqslant 1}q_k\int _0^t e^{-(t-s)}ds=1.
\end{align*}
It follows that $\Phi_u:X_T\rightarrow X_T$ is well defined.
Let us now choose $u,v\in X_T$. We have  that
\begin{align*}
\|\Phi_u(t)-&\Phi_v(t)\|_{L^\infty(\rr^d)} \leqslant \int _0^t e^{-(t-s)} \sum _{k\geqslant 1}q_k\|u^k(s)-v^k(s)\|_{L^\infty(\rr^d)}ds\\
&\leqslant \sum _{k\geqslant 1} kq_k \int _0^t e^{-(t-s)}\|u(s)-v(s)\|_{L^\infty(\rr^d)}ds\leqslant (1-e^{-t}) q \|u(s)-v(s)\|_{X_T}.
\end{align*}
Hence, it follows that there exists a time $T_0=T_0(q)$ such that $\Phi_u$ is a contraction on $X_{T_0}$. 
Then there exists a unique solution $u$ of the fixed point problem $u=\Phi_u$ in the set $X_{T_0}$. 
Since $T_0$ can be fixed independently of the initial data $\varphi$ and $u(T_0)$ 
belongs to $ L^\infty(\rr^d)$ with $0\leqslant u(T_0)\leqslant 1$ we can repeat the same procedure to obtain a global solution.

When $\varphi$ also belongs to $L^1(\rr)$ the same argument as above allows to construct a solution in $C([0,T],L^1(\rr)\cap X_T)$. 
Since we have uniqueness in $X_T$ this implies that the new obtained solution is exactly the  one constructed previously in $X_T$.

Let us now prove the other two properties of $u$ given in the theorem. 
Estimates \eqref{prop.2} and \eqref{prop.4} shows the global existence of solutions of eequation \eqref{sys.u} in $L^1(\rr^d)\cap L^\infty(\rr^d))$.

The second one concerning the mass of the solution is immediate since $u\in [0,1]$ and $\sum _{k\geqslant 1}q_k=1$:
\[
\frac {d}{dt}\int _{\rr^d}u(x,t)\, dx=\int _{\rr^d} \Big(-u (x,t)+\sum _{k\geqslant 1} q_k u^k(x,t) \Big)\, dx\leqslant 0.
\]

In order to prove \eqref{prop.3} we consider the case of $L^1(\rr)$-solutions since once \eqref{prop.3} is obtained for such solutions then by an approximation argument we can extend it to $L^\infty(\rr)$-solutions.
Let us denote by $M=\|\varphi\|_{L^\infty(\rr^d)}\leqslant 1$. Then
\begin{align*}
\frac {d}{dt}\int _{\rr^d}(u(t)-M)^+dx&=\int _{\rr^d} u_t \sgn (u(t)-M)^+dx\\
&=\int _{\rr^d} \Delta u (u(t)-M)^+dx +
\int _{\rr^d}( \sum _{k\geqslant 1} q_k u^k-u)\sgn (u(t)-M)^+dx\\
&\leqslant \int _{M\leqslant u\leqslant 1}( \sum _{k\geqslant 1} q_k u^k-u)dx\leqslant 0.
\end{align*}
Hence for any $t>0$,  $(u(t)-M)^+=0$ and we conclude that $u(t)\leqslant M$.

Let us now prove the contraction property \eqref{prop.4}. Using  property \eqref{s.1} of the semigroup $S(t)$ and the fact that $0\leqslant u_1, u_2\leqslant 1$ we have
\begin{align*}
\label{}
  \|u_1(t)-u_2(t)\|_{L^1(\rr^d)}&\leqslant e^{-t}\|\varphi_1-\varphi_2 \|_{L^1(\rr^d)}+\sum _{k\geqslant 1} \int _0^{t} e^{-(t-s)}q_k\| u_1^k(s)-u_2^k(s)\|_{L^1(\rr^d)}ds\\
  &\leqslant  e^{-t}\|\varphi_1-\varphi_2 \|_{L^1(\rr^d)}+\sum _{k\geqslant 1} kq_k \int _0^{t} e^{-(t-s)}\| u_1(s)-u_2(s)\|_{L^1(\rr^d)}ds.
\end{align*}
Hence, using Gronwall's lemma we obtain that
\[
   \|u_1(t)-u_2(t)\|_{L^1(\rr^d)}\leqslant \|\varphi_1-\varphi_2 \|_{L^1(\rr^d)} e^{(q-1)t}  \mbox{ for all } t\geqslant 0,
 \]
as we wanted to show.
\end{proof}

We now turn our attention to the asymptotic behavior of the solutions as $t\to \infty$ and prove Theorem \ref{asymptotic}.

\begin{proof}[Proof of Theorem \ref{asymptotic}]
\textbf{Step I. Proof of estimate \eqref{est.1}.}
We first observe that the solution $u$ of \eqref{sys.u} is a subsolution for the heat equation, 
that is, 
since
$$F(u)=-u+\sum _{k\geqslant 1}q_ku^k\leqslant u \Big(-1+\sum _{k\geqslant 1} q_k\Big)= 0,$$
we have that
$u$ verifies,
\begin{equation*}
\left\{
\begin{array}{ll}
\displaystyle u_t (x,t)\leqslant \Delta u (x,t), \qquad & x\in \rr^d, t>0,\\[10pt]
u(x,0)=\varphi(x).
\end{array}
\right.
\end{equation*}
Also $u$ is a supersolution for the linear  equation
\begin{equation*}
\left\{
\begin{array}{ll}
\displaystyle u_t(x,t)\geqslant \Delta u (x,t) -u (x,t) +q_1u (x,t), \qquad & x\in \rr^d, t>0,\\[10pt]
u(x,0)=\varphi (x).
\end{array}
\right.
\end{equation*}
By a comparison argument, it follows that
\begin{equation}\label{est.u.5}
e^{-(1-q_1)t}(K_t\ast \varphi)(x) \leqslant u(x,t)\leqslant (K_t\ast \varphi)(x) \mbox{ for all }  x\in \rr^d,  t >0.
\end{equation}
We now improve the right hand side of the above estimate by using that $u$ satisfies
\begin{equation}
\label{improve.u}
  u_t \leqslant  \Delta u +u (q_1-1)+u^p\sum _{k\geqslant p}q_ku^{k-p}\leq\Delta u -u (1-q_1)+u^p.
\end{equation}
Since the initial data satisfies $\varphi\in L^1(\rr^d)$ the right hand side of \eqref{est.u.5} gives us
$$u(x,t)\leqslant \min\{1, ( 4\pi t)^{-d/2}\|\varphi\|_{L^1(\rr^d)}\}.$$
Hence there exists a positive constant $C=C(d,\|\varphi\|_{L^1(\rr^d)})$ such that
\begin{equation}\label{est.u.1}
u(x,t)\leqslant  \frac{C} {(1+t)^{d/2} }\mbox{ for all }   t>0 \mbox{ and }  x\in \rr^d.
\end{equation}
In the following $C$ is a constant that depends on $d$, $p$ and $\|\varphi\|_{L^1(\rr^d)}$ but may change from one line to another. 
Using estimates \eqref{improve.u} and  \eqref{est.u.1} it follows that $u$ satisfies
\begin{align*}
u_t \leqslant \Delta u -u (1-q_1)+C u (1+t)^{-\frac d 2(p-1)} .
\end{align*}
This implies that $$v(t)=e^{(1-q_1)t}u(t)$$ satisfies
$v_t\leqslant \Delta v+va(t) $
where $$a(t)=C(1+t)^{-\frac d 2(p-1)}.$$
Denoting 
$H(t)=\exp(-\int _0^t a(s)ds)$
it  follows that $Hv$ is a sub-solution for the classical  heat equation
 and then  $v$ satisfies $$v(t)\leqslant (K_t\ast \varphi ) \exp(\int_0^t a(s)ds).$$ 
Since $d\geqslant 1$ and $p\geqslant 2$ are integers it follows that $d(p-1)/2\geqslant 1/2$ hence  
 \[
\int _0^t a(s) ds\leqslant C \int _0^t (1+s)^{-1/2}ds=C (1+t)^{1/2}.
\]
 It follows that   $u$ satisfies
 \begin{align}\label{est.u.2}
u(x,t)&\leqslant (K_t\ast \varphi)(x) e^{-(1-q_1)t+C\sqrt{t+1} }\leqslant C\exp(-\frac{(1-q_1)t}2), \quad \forall\, t\geq 0.
\end{align}
 Introducing this estimate in \eqref{improve.u}  we obtain that $u$ satisfies
 \[
u_t \leqslant \Delta u +u (q_1-1)+Cu \exp \Big(-\frac{(1-q_1)(p-1)t}2 \Big), \quad \forall\, t\geq 0 .
\]
Denoting $$a_1(t)=C \exp \Big(-\frac{(1-q_1)(p-1)t}2\Big)$$ and repeating the above argument we get
\begin{align*}
u(x,t)\leqslant (K_t\ast \varphi)(x) e^{-(1-q_1)t} \exp
\Big(\int _0^t a_1(s)ds \Big)\leqslant Ce^{-(1-q_1)t}  (K_t\ast \varphi)(x).
\end{align*}
The proof of estimate  \eqref{est.1} is now finished.

\medskip
 
 \textbf{Step II. Proof of estimate \eqref{est.2}. } Recall  that $v$ is given by
 $$v(t)=e^{t(1-q_1)}u(t).$$
In the following we  prove that
\begin{equation}\label{est.v.1}
\lim _{t\rightarrow \infty} t^{d/2}\|v(t)-C_\varphi K_t\|_{L^\infty (\rr^d)}=0.
\end{equation}
In view of Step I, the function $v$  satisfies
\begin{equation}\label{est.v}
K_t\ast \varphi \leqslant v(t) \leqslant C(d,p,\|\varphi\|_{L^1(\rr)}) K_t\ast \varphi.
\end{equation}
Moreover, it verifies the equation
\begin{equation}\label{sys.v}
\left\{
\begin{array}{ll}
\displaystyle v_t (x,t)=\Delta v  (x,t) +\sum _{k\geqslant p}q_k e^{-t(1-q_1)(k-1)} v^k  (x,t), \qquad & x\in \rr^d, t>0,\\
v(x,0)=\varphi (x).
\end{array}
\right.
\end{equation}
The mass of 
 $v$ satisfies 
\begin{equation}\label{mass.v}
\frac{d}{dt}\int _{\rr^d}v(x,t)\, dx=\int _{\rr^d}\sum _{k\geqslant p} q_k e^{-t(1-q_1)(k-1)} v^k(x, t)\, dx\geqslant 0.
\end{equation} 
 and in view of \eqref{est.v} 
 is uniformly bounded:
$$\int _{\rr^d}v(x,t)\, dx\leqslant C(d,p,\|\varphi\|_{L^1(\rr)})  \|\varphi\|_{L^1(\rr^d)}.$$
Hence, there exists a positive constant $C_\varphi\geqslant \|\varphi\|_{L^1(\rr^d)}$, such that
$$\lim _{t\rightarrow \infty} \int _{\rr^d}v(t)dt=C_\varphi.$$
Observe that in view of \eqref{mass.v} the following also holds:
$$\int _0^\infty \int _{\rr^d}\sum _{k\geqslant p} q_k e^{-t(1-q_1)(k-1)} v^k(s)ds=C_\varphi-\int _{\rr^d} \varphi.$$

We now prove \eqref{est.v.1}.
Let us fix $\varepsilon>0$ and then choose $t_0$ large enough such that
$$C_\varphi- \int_{\rr} v(x,t_0)dx=\int _{t_0}^\infty \int _{\rr^d}\sum _{k\geqslant p} q_k e^{-t(1-q_1)(k-1)} v^k(s)ds<\eps.$$
It follows that 
\begin{align*}
\|v(t)-C_\varphi K_t\|_{L^\infty(\rr^d)}&\leqslant \|v(t)- K_t\ast v(t_0)\|_{L^\infty(\rr^d)}+
\Big\| K_t\ast v(t_0)-K_t \int _{\rr}v(x,t_0)dx\Big\|_{L^\infty(\rr^d)}\\
&\qquad + \|K_t\|_{L^\infty(\rr^d)} \left|C_\varphi -\int _{\rr}v(x,t_0)dx \right|.
\end{align*}
Using the properties of the heat kernel (see for example Lemma 3 in \cite{escobedo}) we have that
\[
t^{d/2}\Big\| K_t\ast v(t_0)-K_t \int _{\rr}v(x,t_0)dx\Big\|_{L^\infty(\rr^d)} \rightarrow 0, \qquad \text{as}\, t\rightarrow \infty.
\]
It remains to prove that for $t$ large enough the following holds
\[
t^{d/2} \|v(t)- K_t\ast v(t_0)\|_{L^\infty(\rr^d)}\leqslant \eps.
\]
Since $v$ is solution of \eqref{sys.v} 
 for any $t>2t_0$, we have
\begin{align*}
\|v(t)-K_t\ast v(t_0)\|_{L^\infty(\rr^d)}&\leqslant \int _{t_0}^t \sum _{k\geqslant p} q_ke^{-s(1-q_1)(k-1)}  \|K_{t-s} \ast  v^k(s)\|_{L^\infty(\rr^d)}ds\\
&= \int _{t_0}^{t/2}+\int _{t/2}^t =I_1(t)+I_2(t).
\end{align*}
For $I_1$ we have
\begin{align*}
I_1(t) & \lesssim  \int _{t_0}^{t/2} \sum _{k\geqslant p} q_k\frac1{(t-s)^{d/2}}  e^{-s(1-q_1)(k-1)} \|v^k(s)\|_{L^1(\rr^d)}ds\\
&\lesssim t^{-d/2}\int _{t_0}^{t/2} \sum _{k\geqslant p} q_k e^{-s(1-q_1)(k-1)} \|v(s)\|^k_{L^k(\rr^d)}ds.
\end{align*}
Estimate \eqref{est.v} shows that for $t_0$ large enough we have
$\|v(s)\|_{L^q(\rr^d)}\leqslant 1$
for any $s\geq t_0$ and $q>1 $. Thus
\begin{align*}
I_1(t)& \leq
t^{-d/2}\sum _{k\geqslant p} \int _{t_0}^{t/2} q_k  e^{-s(1-q_1)(k-1)} ds \\
&\lesssim  t^{-d/2}  \sum _{k\geqslant p}  q_k  e^{-t_0(1-q_1)(k-1)}\leqslant t^{-d/2} e^{-t_0(1-q_1)} 
\\
& \leqslant \eps t^{-d/2}.
\end{align*}
In the case of  $I_2$ we use that $\|v(s)\|_{L^q(\rr^d)}\leqslant 1$
for any $s\geq t_0$. Then
\begin{align*}
I_2(t)&\leqslant \int _{t/2}^t \sum _{k\geqslant p}q_k e^{-s(1-q_1)(k-1)}ds\lesssim  e^{-t(1-q_1)/2}.
 \end{align*}
 The proof is now finished.
\end{proof}

\subsection{Properties of the asymptotic mass}

In this section we analyze  some properties of the asymptotic mass $C_\varphi$:
\[
C_\varphi=\lim _{t\rightarrow \infty} e^{(1-t)q_1}\int _{\rr^d}u(x,t)\, dx.
\]
To do that we need 
 the following comparison principle for the solutions of our problem.
We say that $\supu\in C([0,T],L^1(\rr^d)\cap L^\infty(\rr^d))$ is a supersolution for problem \eqref{sys.u} if $\supu$ satisfies
\begin{equation}\label{sys.supu}
\left\{
\begin{array}{ll}
\displaystyle \supu_t (x,t)\geqslant \Delta \supu (x,t) -\supu (x,t) +\sum _{k\geqslant 1}q_k\supu^k (x,t),
\qquad & x\in \rr^d, t>0,\\
0\leqslant \supu (x,t)\leqslant 1,\\[10pt]
\supu(x,0)\geqslant \varphi (x).
\end{array}
\right.
\end{equation}
In a similar way we  define the subsolution $\subu$.

\begin{theorem}\label{sub-super}
Any sub/super-solutions $\subu$ and $\supu$ that belongs to $C([0,T],L^1(\rr^d)\cap L^\infty(\rr^d))$ satisfy $\subu(t)\leqslant \supu(t)$ on $[0,T]$.
\end{theorem}

\begin{proof}
Consider $\alpha$ a positive number that will be chosen latter. The following holds:
\begin{align*}
\frac {d}{dt}\int _{\rr^d}e^{-\alpha t}&(\subu(t)-\supu(t))^+ =\int _{\rr^d} (\subu_t-\supu_t)\sgn (\subu-\supu)^+e^{-\alpha t}-\alpha\int _{\rr^d}(u-v)^+e^{-\alpha t} \\
&=e^{-\alpha t}\int _{\rr^d} (\Delta \subu-\Delta \supu)\sgn (\subu-\supu)^+\\
&\quad+e^{-\alpha t}\Big[\int _{\rr^d} (f(\subu)-f(\supu))\sgn (\subu-\supu)^+-\alpha \int _{\rr^d}(\subu-\supu)^+\Big]\\
&\leqslant e^{-\alpha t} \Big[  \int _{\subu>\supu} (f(\subu)-\alpha \subu)-(f(\supu)-\alpha \supu)\Big] ,
\end{align*}
where $$f(u)=-u +\sum _{k\geqslant 1}q_ku^k.$$
Since  $\subu$ and $\supu$ belongs to the interval $[0,1]$ we can choose a positive $\alpha$ such that the map $u\rightarrow f(u)-\alpha u$ is decreasing in the interval $[0,1]$. It follows that the right hand side of the above estimate is negative. Thus
$\subu(t)\leqslant \supu(t)$. The proof is now finished.
\end{proof}

 The first property of the asymptotic mass says that the map $\varphi\in L^1(\rr^d)\mapsto C_\varphi$ is contractive and the second one involves the convexity of this map.

\begin{theorem}\label{cont+convexity} 
The map $\varphi\in L^1(\rr^d)\mapsto C_\varphi$ satisfies the following properties:

1. For any $\varphi_{1}, \varphi_2\in L^1(\rr^d)$ with $0\leqslant \varphi_1,\varphi_2\leqslant 1$ there exists a positive constant
$C=C(d,q,p,\|\varphi_1\|_{L^1(\rr^d)}, \|\varphi_2\|_{L^1(\rr^d)})$ (the constant remains bounded in the balls of radius $R$ of $L^1(\rr^d)$)
 such that
the following holds
\begin{equation}\label{est.contractie}
|C_{\varphi_1}-C_{\varphi_2}|\leqslant C \|\varphi_1-\varphi_2\|_{L^1(\rr^d)}.
\end{equation}

2. For any $\lambda\in (0,1)$ and any $\varphi_{1}, \varphi_2\in L^1(\rr^d)$ with $0\leqslant \varphi_1,\varphi_2\leqslant 1$ we have
\begin{equation}\label{est.convexity}
C_{\lambda \varphi_1+(1-\lambda)\varphi_2 }\leqslant \lambda C_{\varphi_1}+ (1-\lambda) C_{\varphi_2}.
\end{equation}
\end{theorem}

As a consequence of the above theorem we have that the map is continuous, in particular 
for any sequence $\varphi_n\in L^1(\rr)$, $0\leqslant \varphi_n\leqslant 1$ such that $\varphi_n\rightarrow \varphi$ in $L^1(\rr)$ we have
$C_{\varphi_n}\rightarrow C_{\varphi}$.

\begin{proof}
Let $u_1$ and $u_2$ be solutions to \eqref{sys.u} corresponding to the initial data $\varphi_1$ and $\varphi_2$. Set $z=v_1-v_2=e^{t(1-q_1)}(u_1-u_2)$.  Then $z$ satisfies 
\begin{equation}\label{sys.z}
\left\{
\begin{array}{ll}
\displaystyle z_t (x,t)=\Delta z  (x,t)+\sum _{k\geqslant 1}q_k e^{-t(1-q_1)(k-1)}(v_1^k-v_2^k) (x,t), \qquad & x\in \rr^d, t>0,\\
z(x,0)=\varphi_1(x)-\varphi_2(x).
\end{array}
\right.
\end{equation}
Property \eqref{prop.4} in Theorem \ref{existence} shows that 
\begin{equation}\label{est.on.bounded.time}
\|z(t)\|_{L^1(\rr^d)}\leqslant C(t,q)\|\varphi_1-\varphi_2\|_{L^1(\rr^d)}.
\end{equation}
We  need to improve this estimate since for large $t$ the constant obtained in the proof of  Theorem \ref{existence}  blows-up when the time goes to infinity.
Observe that we have
\[
|v_1^k-v_2^k|\leqslant k |v_1-v_2|\max\{\|v_1\|_{L^\infty(\rr)}, \|v_2\|_{L^\infty(\rr)},\}\leqslant \tilde{C}^{k-1} k |v_1-v_2|.
\]
for some constant $ \tilde{C}= \tilde{C}(\|\varphi_1\|_{L^1(\rr^d)}, \|\varphi_2\|_{L^1(\rr^d)}, d, p)$ given by estimate \eqref{est.1} in Theorem \ref{asymptotic}.
Thus for any $t\geqslant t_0$ we have
\[
\frac{d}{dt}\int _{\rr^d}|z(x,t)|dx\leqslant |v_1(t)-v_2(t)| \Big(\sum _{k\geqslant p} kq_k \tilde C^{k-1}e^{-t(1-q_1)(k-1)} \Big ),
\]
where  $t_0$  is large enough such that the last term in the right hand side is finite. 
Applying Gronwall's  inequality we get
\[
\|z(t)\|_{L^1(\rr^d)}\leqslant \|z(t_0)\|_{L^1(\rr^d)}\exp(A)
\]
where 
\[
A=\int_{t_0}^\infty   \sum _{k\geqslant p} q_k \tilde C^{k-1}e^{-t(1-q_1)(k-1)}dt=\sum_{k\geqslant p} \frac{kq_k}{(k-1)(1-q_1)} (\tilde C  e^{-t_0(1-q_1)})^{k-1}\leqslant 1
\]
provided $\tilde C\leqslant e^{t_0(1-q_1)}$.
Using estimate \eqref{est.on.bounded.time} with $t=t_0$ we finally obtain that for any $t>t_0$ the following holds
\[
\|z(t)\|_{L^1(\rr^d)}\leqslant C(t_0,q)\|\varphi_1-\varphi_2\|_{L^1(\rr^d)}.
\]
Letting now $t\rightarrow \infty$ we obtain that
\[
|C_{\varphi_1}-C_{\varphi_2}|\leqslant  {C}( d, p,q,\|\varphi_1\|_{L^1(\rr^d)}, \|\varphi_2\|_{L^1(\rr^d)}) \|\varphi_1-\varphi_2\|_{L^1(\rr^d)}
\] 
which proves the first part of the theorem.

Let us now consider the convexity property of the map $\varphi\mapsto C_{\varphi}$. Let us denote by $w$ the solution of
\begin{equation}\label{sys.w.2}
\left\{
\begin{array}{ll}
\displaystyle w_t (x,t)=\Delta w (x,t) -w (x,t) +\sum _{k\geqslant 1}q_kw^k (x,t), \qquad & x\in \rr^d, t>0,\\
w(x,0)=\lambda \varphi_1 (x)+(1-\lambda)\varphi_2 (x),
\end{array}
\right.
\end{equation}
and by $u_1$ and $u_2$ the solutions of \eqref{sys.u} corresponding to initial data $\varphi_1$ and $\varphi_2$. Let $$U=\lambda u_1+(1-\lambda)u_2.$$ This new function $U$ satisfies
\begin{align*}
U_t&=\Delta U-U +\sum _{k\geqslant 1}q_k \Big( \lambda u_1^k +(1-\lambda) u_2^k\Big) \\
&\geqslant 
\Delta U-U +\sum _{k\geqslant 1}q_k ( \lambda u_1 +(1-\lambda) u_2)^k \\
&= \Delta U-U +\sum _{k\geqslant 1}q_k  U^k.
\end{align*}
This show that $U$ is a super-solution for system \eqref{sys.w.2} and by maximum principle given by Theorem \ref{sub-super} we obtain $U\geqslant w$. Hence,
\[
\int _{\rr} w(x,t)dx\leqslant \lambda \int _{\rr} u_1(x,t)dx+(1-\lambda)\int _{\rr} u_2(x,t)dx.
\]
Multiplying the above inequality with $e^{(1-q_1)t}$ and letting $t\rightarrow \infty$ we get
\[
C_{\lambda \varphi_1+(1-\lambda)\varphi_2 }\leqslant \lambda C_{\varphi_1}+ (1-\lambda) C_{\varphi_2}.
\]
The proof is now complete.
\end{proof}

\section{Branching processes on the finite configurations of $\R^d$} 
\label{sect.3}

Let $E:= \R^d$ and
define the set  $\widehat{E}$ of finite positive measures on $E$ as
$$
\widehat{E}:=\left\{\sum_{1\leq k{\leqslant}k_0}\delta_{x_k}:k_{ 0}\in\N^*, x_k\in E\textrm{ for all } 1{\leqslant}k{\leqslant}k_0\right\}\cup\{{\bf 0}\},
$$
where ${\bf 0}$ denotes the zero measure. 
The set $\widehat{E}$ is identified with the union of all symmetric $m$-th powers $E^{(m)}$ of $E$,
$\widehat{E}= \bigcup_{m {\geqslant}0}E^{(m)},$
where $E^{(0)}:=\{\bf 0\}$, and it is called the {\it space of finite configurations of E}.
 $\widehat{E}$  is endowed with the topology of disjoint union of topological spaces 
and the corresponding Borel $\sigma$-algebra is denoted by $\cb(\widehat{E})$;
see, e.g.,  \cite{INW68, BeOp11, BeLu14}. 

A Markov process $\overline{X}$ with state space $\widehat{E}$ is called {\it  branching process} 
provided that for all $\mu_1, \mu_2 \in \widehat{E}$, the process $X^{\mu_1+\mu_2}$ starting from $\mu_1+\mu_2$
and the sum $X^{\mu_1} + X^{\mu_2}$ are equal in distributions.\\

\noindent
{\bf Branching kernel on $\widehat{E}$}.
The  {\it convolution} of two finite measures $p_1, p_2$  on $\widehat{E}$
is the finite measure $p_1*p_2$ on $\widehat{E}$ defined by
$$
\int p_1*p_2(d\nu)\bff (\nu) := \int p_1(d\nu_1) \int p_2(d\nu_2)
\bff (\nu_1+\nu_2), \qquad  \bff  \in bp\cb(\widehat{E}).
$$
 
 A {\it branching kernel}  is a kernel $N$ on $\widehat{E}$
such that for all  $\mu, \nu\in \widehat{E}$  we have
$N_{\mu+\nu}= N_\mu * N_\nu ,$
where $N_\mu$ denotes  the measure on  $\widehat{E}$
such that $N \bff (\mu)=\int \bff  \, d N_\mu$ for all
$\bff  \in bp\cb(\widehat{E})$.

Recall that a  Markov process  with state space $\widehat{E}$ is a branching process
if and only if its transition function if formed from branching kernels.\\

\noindent
{\bf Example; the diagonal kernel on $\widehat{E}$. }
Let $N$ be a kernel on $E$ which is sub-Markovian (i.e. $N1\leqslant 1$) 
and for every $k \geqslant 1$ consider the kernel $N^k$ on $E^k$, 
the $k$-times product of  $N$,  defined as
$$
N^k f(x):=\underbrace{\int \ldots\int}_{k\mbox{\small -times}}
f(y_1,\ldots,y_k)N(x_1,dy_1)\ldots N(x_k,dy_k)
$$
for all  $f\in p\cb(E^k)$ and $x=(x_1,...,x_k)\in E^k.$
Let now ${\bf N}$ be the sub-Markovian  kernel on $\widehat{E}$ defined as
$$
{\bf N} \bff := \sum_{k \geqslant 0} N^{(k)} (\bff |_{E^{(k)}}),\quad  \bff \in bp\cb(\widehat{E}),
$$
where $N^{(k)}$, $k\geqslant 1$  denotes the symmetric $k$-times power of $N$, that is the projection on $E^{(k)}$ 
of the  kernel $N^k$ on $E^k$, and $N^{(0)}:= \delta_{\mbox{\bf\small 0}}$; see \cite{BeOp11}, Section 4.1 for details.
The kernel $\bf N$ is called {\it diagonal} and one can check  that it is a branching kernel on $\widehat{E}$.\\

We present now two branching processes on the space of finite configurations of $E$.
First,  we  indicate canonical constructions of path continuous Markov processes on $\widehat{E}$, induced by the $d$-dimensional Brownian motion
(cf.,  e.g., \cite{BeOp11}).

Let $B= (B_t)_{t\geqslant 0}$ be the $d$-dimensional Brownian motion and $(P_t)_{t\geqslant 0}$ be its transition function,
the $d$-dimensional Wiener semigroup on $\R^d$.

For $k \geqslant 1$, let $B^k= (B^k_t)_{t\geqslant 0}$ be the $kd$-dimensional Brownian motion, regarded as a path continuous Markov process
with state space $E^k=\R^{kd}$. 
Recall that $B^k$ is the $k$ times Cartesian power of the $d$-dimensional Brownian and let 
$(P_t^k)_{t\geqslant 0}$ be its transition function, the $kd$-dimensional Wiener semigroup.
Let $\cu^k=(U_{\alpha}^k)_{\alpha>0}$ be its resolvent of kernels on $E^k$, 
$$
U_{\alpha}^k f(x)=\E^x \int_{0}^{\infty} e^{-\alpha t} f(B^k_t)  dt,\qquad f \in
p\cb(E^k), x\in E^k.
$$

For each $k\in\N^{*}$, $t\geqslant 0$, and $\alpha>0$, we consider the projections  $P_t^{(k)}$ and $U_{\alpha}^{(k)}$  
on $E^{(k)}$ of the kernels $P_t^k$ and $U_{\alpha}^k$  on $E^{k}$
(for details see \cite{BeOp11}, Section 4.1). 
Namely, by Proposition 4.1 and Proposition 4.2 in \cite{BeOp11} we get that 
the family $(P_t^{(k)})_{t \geqslant 0}$ is a sub-Markovian semigroup of kernels on $E^{(k)}$,
with the  induced resolvent of kernels $\cu^{(k)}=(U_{\alpha}^{(k)})_{\alpha>0}$.
The semigroup of kernels $(P_t^{(k)})_{t \geqslant 0}$ is the transition 
function of a right (Markov) process   $B^{(k)}$
with state space $E^{(k)}$,
the  symmetric $dk$ times power of $B= (B_t)_{t\geqslant 0}$.

Let $({\bf P}_t)_{t\geqslant 0}$ be the family of kernels on $\widehat{E}$ defined as
${\bf P}_t {\bf g}|_{E^{(k)}}:= P^{(k)}_t ({\bf g}|_{E^{(k)}})$, ${\bf g}\in p\cb(\widehat{E})$.
It is the transition function of a Borel right process $\widehat{B}= (\widehat{B}_t)_{t\geqslant 0}$ with state
space $\widehat{E}$, called {\it Brownian motion on  $\widehat{E}$}.
So,
$$
{\bf P}_t {\bf g}(\mu)= \E^\mu {\bf g}(\widehat{B}_t), \ t>0, \  \mu \in \widehat{E}.
$$
Note that since the transition function $({\bf P}_t)_{t\geqslant 0}$ is formed by diagonal kernels 
(in particular, they  are branching kernels), we conclude that
$\widehat{B}= (\widehat{B}_t)_{t\geqslant 0}$ is a branching process on $\widehat{E}$.

If $\varphi\in p\cb(E)$, define the {\it multiplicative function }  $\widehat{\varphi}:\widehat{E}\longrightarrow \R_+$ as
\begin{equation}\nonumber
\widehat{\varphi}({\bf x})=
\left\{
\begin{array}{ll}
\displaystyle\prod_{k\geq 1}\varphi(x_k), \textrm{ if }{\bf x}=(x_k)_{k{\geqslant}1}\in \widehat{E},  & 
\textrm{ if }{\bf x}\not= {\bf 0},  
\\
1, & \textrm{ if } {\bf x}= {\bf 0}.
\end{array}
\right. 
\end{equation}

The next proposition gives  the second example of a branching process on the set of all finite configurations of $\R^d$. 
More precisely,  the solution $u(x,t)$, $t\geqslant 0$, $x\in E$,  of \eqref{sys.u}  given by Theorem \ref{existence}, 
induces a branching  processes on $\widehat{\R^d}$.

\begin{prop} \label{existbranch}  
There exists a branching process  $\widehat{X}= (\widehat{X}_t)_{t\geqslant 0}$ with state space $\widehat{\R^d}$ such that if we denote by 
$({\bf H}_t)_{t\geqslant 0}$  its transition function, then we have
$$
{\bf H}_t \widehat{\varphi}= \widehat{u(\cdot, t)}  \mbox{ for all }\  0\leqslant \varphi \leqslant 1. 
$$
\end{prop}

\begin{proof}
The assertion follows by Theorem 4.10 from \cite{BeLu14}. 
Note that the hypothesis  $(*)$  from that theorem is satisfied in this case, due to Proposition 4.9 from the same paper.
\end{proof}

We can  state now the probabilistic interpretation of Theorem \ref{asymptotic}, 
in terms of branching processes on the space of all finite configurations of $\R^d$.

\begin{theorem} 
 For any $\varphi\in L^1(\rr^d)$, $0\leqslant \varphi\leqslant 1$,  
there exists a positive constant $C(\varphi)$ such that 
$$
t^{d/2} |e^{t(1-q_1)} \mathbf{H}_t \widehat \varphi - \widehat{C(\varphi)} {\bf{P}}_t\widehat \varphi| \leqslant l_1 o(1),
$$
where $({\bf P}_t)_{t\geqslant 0}$ is the transition function of the Brownian motion on $\widehat{\R^d}$, 
induced by the $d$-dimensional Brownian motion.
\end{theorem}

\begin{proof}
Recall first that if $0\leqslant \varphi, \psi\leqslant 1$ then
$$
|\widehat{\varphi}(x)- \widehat{\psi} (x) | \leqslant  l_1 ||\varphi -\psi ||_\infty,\quad \forall x\in \widehat E,
$$
where $l_1$ is the linear functional defined on $\widehat{E}$ as: $l_1|_{E^{(k)}}:= k$, $k \geqslant 0$.
Taking into account that by Proposition \ref{existbranch}
${\bf H}_t \widehat{\varphi}= \widehat{u(\cdot, t)}$ 
and since we also have, 
${\bf P}_t \widehat{\varphi} =\widehat{P_t \varphi}$,
we can use Theorem \ref{asymptotic} to obtain
\[
t^{d/2} |e^{t(1-q_1)} \mathbf{H}_t \widehat \varphi - \widehat{C(\varphi)}  {\bf{P}}_t\widehat \varphi| \leqslant 
l_1 t^{d/2}\|e^{t(1-q_1)}  u(\cdot, t)  - C(\varphi) P_t  \varphi \|_{L^\infty(\rr^d)}\leqslant l_1 o(1).
\]
\end{proof}

\begin{cor} \label{cor3.1}
Let $A\in \cb(\R^d)$, $\mu\in\widehat{\R^d}$, and let $\widehat{X}= (\widehat{X}_t)_{t\geqslant 0}$   
be the   branching process  with state space $\widehat{\R^d}$,  given by Proposition \ref{existbranch}.
Then
$$
|\E^\mu(\widehat X_t \in \widehat A)  - \widehat{C(1_A)} \E^\mu (\widehat B^{1-q_1}_t\in \widehat A)|\leqslant \frac{l_1o(1)}{e^{t(1-q_1)}  \, t^{d/2}},\,  \qquad t >0,
$$
where $\widehat B^{1-q_1}$ is the $(1-q_1)$-subprocess of the
Brownian motion $\widehat B$ on $\widehat{\R^d}$.
\end{cor}

\begin{remark}   \label{rem3.1}
{\rm
$(i)$ Corollary \ref{cor3.1} gives the claimed  probabilistic interpretation of the asymptotic result from Theorem \ref{asymptotic}.
In fact, the long time asymptotic  behavior of the distribution of the branching process on the space of finite configurations of $\R^d$,   associated with  the sequence $(q_k)_{k\geqslant 1}$ and  having as base process the $d$-dimensional Brownian motion, is the distribution of $(1-q_1)$-subprocess of the diagonal process,  induced by the $d$-dimensional Brownian motion, multiplied with the corresponding asymptotic mass. 
Note that for the considered  nonlinear equation the asymptotic mass is not just the mass of the initial datum as it happens in the linear case.

$(ii)$ Recall that in the Introduction  we stated the role of the sequence $ (q_k)_{k\geqslant 1}$ in  the probabilistic interpretation of 
the branching mechanism of the process $\widehat{X}= (\widehat{X}_t)_{t\geqslant 0}$ from Propostion \ref{existbranch}.
It is possible to consider different, more general, branching mechanisms; see, e.g.,  \cite{BeLu14} and \cite{BeLuOp12}.
In particular, the statement of Theorem \ref{th1.1} and its proof should be compared with Proposition 4.1 from \cite{BeLu14}.

$(iii)$ Branching processes mentioned in the above assertion $(i)$ were constructed in 
\cite{BeDeLu15} and   \cite{BeDeLu16},  related to the  fragmentation equation
and   respectively to a probabilistic model of the fragmentation phase of an avalanche.

$(iv)$ A nonlinear Dirichlet problem associated to the equation (1) is considered and solved in  
\cite{BeOp11} and \cite{BeOp14}, while a probabilistic numerical approach is developed  in \cite{LuSt17}.
}
\end{remark}

\section{Occupation time}\label{sect.4}  

In  Theorem 1.1 we obtained that for any  $\varphi \in b\mathcal{B}(\mathbb{R}^d)$, $0 \leqslant  \varphi \leqslant  1$,  there exists a unique solution that in the following we denote by $(U_t\varphi)_{t\geq 0}$,  of  equation \eqref{sys.u} that satisfies and $$0 \leqslant  U_t(\varphi) \leqslant  1.$$
Moreover by Proposition \ref{existbranch} it introduces a branching process  $\hat X$ on $\widehat {\R^d}$.
For $f \in bp \mathcal{B}(\mathbb{R}^d)$ we define the function $\tu_t f \in p \mathcal{B}(\mathbb{R}^d)$ as
$$
\tu_t f := -\ln U_t(e^{-f}).
$$
Then by Corollary 4.3, Theorem 4.10 and Remark 4.4 (iii) from \cite{BeLu14} the following assertions hold:

(i) The family $(\tu_t)_{t \geqslant  0}$ is a nonlinear semigroup on $bp \mathcal{B}(\mathbb{R}^d)$ and $\tu_t f$ is the solution of the equation
\begin{equation}
\label{4.2}\left \{
\begin{array}{l}
\displaystyle  \tu_t = \Delta \tu -  |\nabla \tu |^2 + 1 - \sum_{k \geqslant  1} q_k e^{(1 - k)\tu}, \quad t >  0, x\in \rr^d,\\[10pt]
  \tilde U(0)=f.
  \end{array}
\right.
\end{equation}

(ii) For any $ s \geqslant  0, \mu \in \widehat{\mathbb{R}^d}, g \in bp
\mathcal{B}(\mathbb{R}^d) $ the following holds 
\begin{equation}
\label{4.3}
  \mathbb{E}^\mu \exp (-\langle g, \widehat X_s \rangle) = \exp (-\langle \tu_s g, \mu  \rangle).
\end{equation}

Following \cite{Iscoe86}, for every $\varphi \in b\mathcal{B}(\mathbb{R}^d)$, we define the {\it weighted occupation time process} 
$Y_t(\varphi) : \Omega \to \overline{\mathbb{R}}$ as 
$$
Y_t(\varphi) : = \int_0^t\langle \varphi, \widehat X_s\rangle \, ds, \quad t \geqslant  0,
$$
where $\Omega$ is the path space of $\widehat{X}$.
Observe that it is a.s. a real-valued process. 
Indeed, using Proposition 4.8 from \cite{BeLu14} we have
\begin{align*}
\label{}
  \mathbb{E}^\mu |Y_t(\varphi)| & \leqslant  
\mathbb{E}^\mu Y_t (| \varphi |) = 
\int_0^t \mathbb{E}^\mu l_{|\varphi|}(\widehat X_s) ds =
\int_0^t {\bf H}_s l_{|\varphi|} (\mu) ds \\
 &\leqslant\|\varphi\|_\infty \int_0^t {\bf H}_s l_1(\mu) ds \leq
l_1(\mu) \|\varphi\|_\infty \int_0^t e^{(1-q_1) s} ds < \infty.
\end{align*}

Regarding the process $Y_t(\varphi)$ we have the following representation. 

\begin{theorem}  \label{th4.1} 
Let us assume that $\sum _{k\geq 1} k^2 q_k<\infty$. For any 
 $\mu \in \widehat{\mathbb{R}^d}$, $t \geqslant  0$, $f \in bp \mathcal{B}(\mathbb{R}^d)$, 
and $\varphi \in C_b(\mathbb{R}^d)$, $\varphi \geqslant  0$ the following holds
$$
 \mathbb
{\E}^\mu [ \exp (-\langle f, \widehat{X}_t\rangle - Y_t(\varphi))]=  \exp(-\langle 
\tv_t (\varphi, f), \mu \rangle),
$$
where $\tv_t(\varphi, f)$, $t \geqslant  0$ is  the solution to 
\begin{equation}
\label{4.4}
\left\{
\begin{array}{l}
\displaystyle v_t = \Delta v  -  |\nabla v|^2 + 1 - \sum_{k \geqslant  1} q_k e^{(1 - k)v}  +\varphi, 
\qquad t>0, x\in \rr^d, \\[3mm]
v( 0) = f.
\end{array}
\right.
  \end{equation}
\end{theorem}

The above result is an integral representation of the solution $\tv_t(\varphi, f)$ 
with respect to the $\widehat{\mathbb{R}^d}$-valued branching  process $(\widehat{X}_t)_{t \geqslant  0}$ 
and its weighted occupation time process. 
This is a version for  (non local) branching  processes on the space of finite configurations of a result for 
$M(\mathbb{R}^d)$-superprocesses  from \cite{Iscoe86}, Theorem  3.1.
 For the reader convenience we include in Appendix \ref{A2} a sketch of the proof.

\begin{theorem}\label{as.oc.time} 
For any nonnegative and bounded $\varphi \not\equiv 0$, and $T>0$ the solution $v_T$ of the equation 
 \begin{equation}
\label{eq.v.1}
\left\{
\begin{array}{l}
\displaystyle
 v _t= \Delta v  - v+ \sum_{k \geqslant  1} q_k v^k  -\frac 1Tv \varphi, \\[3mm]
v(\cdot, 0) = 1
\end{array}
\right.
  \end{equation} 
satisfies
\begin{equation}
\label{lim.1}
   \lim_{t\rightarrow \infty}v_T(t,x)=0 \mbox{ for all } x\in \rr^d
   \end{equation}
and for any  $\alpha<1/F'(1)$
\begin{equation}
\label{lim.2}
    \lim_{T\to \infty}v_T(\alpha \log T,x)=1  \mbox{ for all } x\in \rr^d.
\end{equation}
\end{theorem}

We now obtain some asymptotic properties of the occupation time. 

\begin{theorem}  
Let us assume that $\sum _{k\geq 1} k^2 q_k<\infty$. 
For any $\varphi\in C_b(\rr^d)$, $\varphi\geq 0$ 
the weighted occupation time process $Y_t$ satisfies
\[
\lim _{t\rightarrow\infty}\E^\mu [\exp(-Y_t(\varphi))] =0, 
\]
and  for $\alpha<(\sum _{k\geq 1}(k-1)q_k)^{-1} $
\[
\lim _{T\rightarrow \infty} \E^\mu  \Big[ \exp (- \frac {Y_T(\varphi)}{\exp(T/\alpha)})\Big]=1.
\]
\end{theorem}

\begin{proof}Let us choose  $F$ to be a finite set of points  $x_k\in \rr^d$ and  $\mu=\sum _{x_k\in F}\delta_{x_k}$. Since we are interested in the asymptotic properties of $Y_t(\varphi)$ we choose $f\equiv 0$ in Theorem \ref{th4.1}.
In the case of the first limit we consider system \eqref{eq.v.1} with $T=1$ and denote by $v$ its soltuion. We have
$$
  \E^\mu [e^{-Y_t(\varphi)} ]=\exp(-\langle \tv_t(\varphi,0),\mu\rangle )
  =\exp\Big(-\sum_{x_k\in F} \tv_t(\varphi,0)(x_k)\Big)=\prod_{x_k\in F} v(t,x_k).
$$
 In view of property \eqref{lim.1} we obtain the  desired result. 
 
 Let us now prove the second limit. Under the same assumptions as before we have
 $$
  \E^\mu \Big[ \exp (- \frac {Y_T(\varphi)}T)\Big]  = \exp\Big(-\sum _{x_k\in F} \tv_T(\frac {\varphi}T,0)(x_k) \Big)=
  \prod_{x_k\in F} v_T(T,x_k),
$$
where $v_T$ is the solution of \eqref{eq.v.1}. Using \eqref{lim.2} we obtain the second property.
\end{proof}

\begin{proof}[Proof of Th.\ref{as.oc.time}] \textbf{Step I. Proof of \eqref{lim.1}.}
We can assume that $x=0$ and prove the required limit. We now consider the truncated problem in the ball $B_R=\{x\in \rr^d\,|x|< R\}$:
	 \begin{equation}
\label{eq.v.truncated}
\left\{
\begin{array}{ll}
\displaystyle v _t= \Delta v  - v+ \sum_{k \geqslant  1} q_k v^k  -\frac 1T v \varphi,  \qquad & (x,t)\in  B_R\times (0,\infty),\\[3mm]
 v(x,t)=1, & (x,t)\in   \partial B_R\times [0,\infty),\\[3mm]
v(x,0) = 1, &x\in B_R.
\end{array}
\right.
  \end{equation} 
  and denote its solution by $v_R$. Using the maximum principle we have 
  \[
  0\leqslant  v(t,x)\leqslant  v_R(t,x)  \mbox{ for all }  t>0, x\in \overline B_R.
\]
Let us now consider the energy functional associated with \eqref{eq.v.truncated}
\[
  E(t)=\int _{B_R}\Big(\frac12 |\nabla v|^2 -\frac12 v^2 +\sum _{k}q_k \frac{v^{k+1}}{k+1}-\frac {v^2\varphi}{2T} \Big)
\]
It is uniformly bounded $-\infty <E(t)\leqslant  E(0)$ for all $t\geqslant  0$ and satisfies 
\[
  \frac {dE}{dt}(t)=-\int _{B_R} v_t^2\leqslant  0.
\]
  Thus $v_R(t)\rightarrow v_R$ in $H^1(B_R)$, as $t\rightarrow\infty$, where $v_R$ is the unique stationary solution of problem \eqref{eq.v.truncated}. Moreover $v_R$ are uniformly bounded in $C^1(\overline B_R)$ and then 
  $v_R(t)\rightarrow v_R$ uniformly in $\overline B_R$ when $t\rightarrow \infty$.
  This implies that for any $\epsilon>0$ there exists $T_\epsilon>0$ such that 
  \begin{equation}
\label{limita.partiala}
  0\leqslant  v(t,0)\leqslant  v_R(t,0)+\epsilon, \mbox{ for all }  t>T_\epsilon.
\end{equation}
The stationary solutions $v_R$ satisfy 
$$
0<v_{R+1}(x)\leqslant  v_R(x)\leqslant  1 \mbox{ for all }  x\in \overline B_R.
$$
  Therefore there exists the following limit 
  \[
  \lim_{R\rightarrow\infty} v_R(x)=v^\infty(x),\quad x\in \rr^d,
\]
and $v^\infty$ solves the problem
   \begin{equation}\label{v.limit}
\left\{
\begin{array}{l}
\displaystyle \Delta v  - v+ \sum_{k \geqslant  1} q_k v^k  -\frac 1Tv \varphi=0, \qquad x\in \rr^d,\\[3mm]
0\leqslant  v\leqslant   1.
\end{array}
\right.
  \end{equation}
Multiplying by $v^\infty$ the above equation  and integrating on $\rr^d$ we obtain that $v^\infty$ should be a constant. 
Since $\varphi\not\equiv 0$ the only constant function that solves \eqref{v.limit} is  $v^\infty\equiv 0$. 
This and  \eqref{limita.partiala} give  us de desired result.


\textbf{Step II. Proof of \eqref{lim.2}.}
	Let us consider $M$ such that $0\leqslant  \varphi\leqslant  M$. Then $v\geqslant  \overline v$ where $\ov$ is solution of 
	 \begin{equation*}
\left\{
\begin{array}{l}
 \displaystyle v _t= \Delta v  - v+ \sum_{k \geqslant  1} q_k v^k  -\frac MTv, \qquad x\in \rr^d, t>0, \\[3mm]
v(\cdot, 0) = 1.
\end{array}
\right.
  \end{equation*}
  The unique solution of this equation depends only on $t$ and then we arrive to the ODE 
  	 \begin{equation*}
\left\{
\begin{array}{l}
 \displaystyle v _t=  - v+ \sum_{k \geqslant  1} q_k v^k  -\frac MTv, \\[3mm]
v(0) = 1.
\end{array}
\right.
  \end{equation*}
  Observe that $$F(v)=- v+ \sum_{k \geqslant  1} q_k v^k\geqslant  F'(1)(v-1)$$ where 
  $$F'(1)=\sum _{k\geqslant  1}kq_k-1=\sum _{k\geq 1}(k-1)q_k >0.$$ 
  This implies that $\ov \geqslant  g_T(T)$ where $g_T$ is the solution of the following ODE:
  	 \begin{equation*}
\left\{
\begin{array}{l}
\displaystyle  g _t=   \Big(F'(1)-\frac MT\Big) g-F'(1), \\[3mm]
 g( 0) = 1.
\end{array}
\right.
  \end{equation*}
Solving explicitly, it  gives us that
\[
  g_T(t)=\frac{TF'(1)}{TF'(1)-M} -\frac {M}{TF'(1)-M}e^{tF'(1)}e^{Mt/T}.
\]
Choosing $t=\alpha \log T$ with  $\alpha<1/F'(1)$ we get
\[
  \lim_{T\rightarrow \infty}g_T(T)=1.
\]
This implies that $v_T$, the  solution of \eqref{eq.v.1}, satisfies
\[
  \lim_{T\rightarrow \infty} v_T(\alpha \log T,x)=1   \mbox{ for all } x\in \rr^d
\]
and the proof is now complete. 
\end{proof}

\section{Appendix} \label{appendix}
\noindent
\subsection{\bf Right Markov processes.} \label{A1}
Let $E$ be a metrizable Lusin topological space 
and $\Bscr$ the Borel $\sigma$-algebra on $E$.

A \emph{transition function} on $E$ is a family $(p_t)_{t\geqslant 0}$
of kernels on $(E,\Bscr)$ which are sub-Markovian
(i.e., \ $p_t 1\leqslant 1$ for all $t\geqslant 0$), such that $p_0 f=f$
and $p_s(p_t f)=p_{s+t} f$ for all $s,t\geqslant 0$ and 
$f\in p\Bscr$.
We assume that for all $f\in p\Bscr$ the function 
$(t,x)\mapsto p_t f(x)$ is 
$\Bscr\bigl([0,\infty)\bigr)\otimes\Bscr$-measurable.
We denote by $\Uscr=(U_\alpha)_{\alpha>0}$ the family
of kernels on $(E,\Bscr)$ given by
\[
  U_\alpha f
  = \int_0^\infty e^{-\alpha t}\, p_t f  d t\;.
\]
Consequently, $\Uscr=(U_\alpha)_{\alpha>0}$ is a sub-Markovian
resolvent of kernels on $(E,\Bscr)$ and it is called
\emph{associated with} $(p_t)_{t\geqslant 0}$.

If  $(p_t)_{t\geqslant 0}$ is a transition function on $E$ and $q>0$, then the family o kernels
$(e^{-qt} p_t)_{t\geqslant 0}$ is also a  transition function on $E$ and its associate resolvent is
$\Uscr_q:=(U_{q+\alpha})_{\alpha>0}$. 
Recall that if $L$ is the generator of $(p_t)_{t\geqslant 0}$, then the generator of
$(e^{-qt} p_t)_{t\geqslant 0}$ is $L-q$.

A \emph{right process with state space $E$} (associated with the
transition function $(p_t)_{t\geqslant 0}$) is a collection
$X=(\Omega,\Gscr,\Gscr_t, X_t,\theta_t, \Pp^x)$ where:
$(\Omega,\Gscr)$ is a measurable space, $(\Gscr_t)_{t\geqslant 0}$
is a family of sub $\sigma$-algebras of $\Gscr$ such that 
$\Gscr_s\subseteq\Gscr_t$ if $s<t$;
for all $t\geqslant 0$, $X_t : \Omega\to E_\Delta$ is a 
$\Gscr_t / \Bscr_\Delta$-measurable map such that 
$X_t(\omega) = \Delta$ for all $t>t_0$ if $X_{t_0}(\omega)=\Delta$,
where $\Delta$ is a cemetery state adjoined to $E$ as an isolated
point of $E_\Delta:= E\cup \{\Delta\}$ and $\Bscr_\Delta$
is the Borel $\sigma$-algebra on $E_\Delta$;
we set
$\zeta(\omega) 
  := \inf \bigl\{ t\bigm| X_t(\omega)=\Delta \bigr\}$;
for each $t\geqslant 0$, the map $\theta_t : \Omega\to\Omega$
is such that $X_s\circ \theta_t=X_{s+t}$ for all $s>0$;
for all $x\in E_\Delta$, $\Pp^x$ is a probability measure
on $(\Omega,\Gscr)$ such that $x\mapsto \Pp^x(F)$ is
universally $\Bscr$-measurable  for all $F\in \Gscr$;
$E^x(f\circ X_0)=f(x)$ and the following 
\emph{Markov property} holds:
\[
  \E^x (f\circ X_{s+t}\cdot G)
  = \E^x (p_t^\Delta f\circ X_s\cdot G)
\]
for all $x\in E_\Delta$, $s,t\geqslant 0$, $f\in p\Bscr_\Delta$
and $G\in p\Gscr_s$, where $p_t^\Delta$ is the Markovian
kernel on $(E_\Delta,\Bscr_\Delta)$ such that $p_t^\Delta 1=1$
and $p_t^\Delta\rs{E} = p_t$;
for all $\omega\in \Omega$ the function $t\mapsto X_t(\omega)$
is right continuous on $[0,\infty)$;
the filtration $(\Gscr_t)_{t\geqslant 0}$ is \emph{right continuous} 
(i.e.\ $\Gscr_t = \Gscr_{t+} := \bigcap\limits_{s>t}\Gscr_s$)
and \emph{augmented} 
(i.e.\ $\Gscr_t=\tilde\Gscr_t := \bigcap_{\mu} \Gscr_t^\mu$,
where for every probability measure $\mu$ on $(E,\Bscr)$,
$\Gscr^\mu$ is the completion of $\Gscr$ with respect to 
the probability measure $\Pp^\mu := \int \Pp^x\;\mu(\mathrm dx)$
on $(\Omega,\Gscr)$ and $\Gscr_t^\mu$ is the completion of 
$\Gscr_t$ in $\Gscr^\mu$ with respect to $\Pp^\mu$);
we assume that for all $\alpha>0$, every function $u$
which is $\alpha$-excessive with respect to the resolvent 
associated with $(p_t)_{t\geqslant 0}$ and each probability 
measure $\mu$ on $(E,\Bscr)$, the function
$t\mapsto u\circ X_t$ is right continuous on $[0,\infty)$
$\Pp^\mu$-a.s.

We consider the \emph{natural filtration} associated with $X$:
$\Fscr := \tilde \Fscr^0$, $\Fscr_t := \tilde \Fscr_t^0$, where
$\Fscr^0 := \sigma (X_s \mid s<\infty)$,
$\Fscr_t^0 := \sigma (X_s\mid s\leqslant t)$.
It is known that always a right process may be considered
with respect to its natural filtration:
\[
  X = (\Omega,\Fscr,\Fscr_t,X_t,\theta_t,\Pp^x)\;.
\]

The sub-Markovian resolvent $\Uscr=(U_\alpha)_{\alpha>0}$
associated with $(p_t)_{t\geqslant 0}$ is called the 
\emph{resolvent of the process $X$} and  for all 
$f\in p\Bscr$, $\alpha>0$ and $x\in E$ we have
$$
  U_\alpha f(x)
  = \E^x \int_0^\zeta e^{-\alpha t}\, f (X_t)\,  d t\, ,
$$
with the convention $f(\Delta)=0$.

If $q>0$ then there exists a right process $X^q$ with state space $E$, 
associated with the transition function $(e^{-qt} p_t)_{t\geqslant 0}$,
called the {\it $q$-subprocess} of $X$. 
The resolvent family of  the process $X^q$ is $(U_{q+\alpha})_{\alpha>0}$.
If $\cl$ is the generator of $X$ (e. g., as the generator of the 
$C_0$-resolvent of contractions induced by $(U_\alpha)_{\alpha>0}$ 
on $L^p(E, m)$, where $m$ is a $(p_t)_{t\geqslant 0}$-subinvariant measure), then the generator of
the {$q$-subprocess}  $X^q$  is $\cl-q$.

A \emph{stopping time} is a map $T: \Omega\to \bar \R_+$
such that the set $[T\leqslant t]$ belongs to $\Fscr_t$ for all
$t\geqslant 0$.

Let $\mu$ be a $\sigma$-finite measure on $(E,\Bscr)$.
The right process $X$ is called \emph{$\mu$-standard},
if it possesses left limits in $E$ $\Pp^\mu$-a.e.\
on $(0,\zeta)$ and for every increasing sequence $(T_n)_n$
of stopping times, $T_n\nearrow T$, the sequence $(X_{T_n})_n$
converges to $X_T$ $\Pp^\mu$-a.e.\ on $[T<\zeta]$.

Let $\Tscr$ be a topology on $E$.
The right process $X$ is named 
\emph{c\`adl\`ag in the topology $\Tscr$ $\Pp^\mu$-a.e.}
provided that $\Pp^\mu$-a.e.\ $t\longmapsto X_t$ is right continuous
and has left limits in $E$ on $(0,\zeta)$.

For more details on right processes see e.g., \cite{Sh88} and \cite{BeBo04}.\\

\subsection{Sketch of the proof of Theorem \ref{th4.1}}\label{A2}
Because $\varphi \in C_b(\mathbb{R}^d)$ we may use the Riemann sum approximation of $Y_t(\varphi) : Y_t(\varphi) = \lim_{N\rightarrow\infty} r(N, t)$,  where
$$
r(kn t) := \sum_{k = 1}^n \langle \varphi, \widehat{X}_{\frac{k}{N }t}\rangle \frac{t}{N}, \quad N \in \mathbb{N}^*, \quad t \geqslant  0.
$$
We prove first that the following holds
\begin{equation}\label{a1}
\mathbb{E}^\mu \exp (-\langle f, \widehat X_t\rangle - Y_t(\varphi))  = \lim_{N\rightarrow\infty} \mathbb{E}^\mu \exp (- r (N - 2, t) - \langle (\tu_{\frac{t}{N}} \tw_{\frac{t}{N}})^2 f, \widehat X_{\frac{N - 2}{N}t}\rangle). 
\end{equation}
Let us denote  $\tw_t f : = f + t \varphi$, where $\varphi \in b \mathcal{B}(\mathbb{R}^d)$, $t \geqslant  0$.
Using  the Markov property we have
$$
\begin{array}{l}
\displaystyle \mathbb{E}^\mu \exp (-\langle f, \widehat X_t\rangle - Y_t(\varphi)) = 
\lim_{N\rightarrow\infty} \mathbb{E}^\mu\left( \exp(-\langle f + \dfrac{t}{N} \varphi, \widehat X_t \rangle) \cdot \exp(-r (N - 1, t))\right) \\[10pt]
\displaystyle \qquad =
\lim_{N\rightarrow\infty} \mathbb{E}^\mu \left( \exp (-\langle \tw_{\frac{t}{N}} f, \widehat X_t\rangle) \cdot \exp (-r (N-1, t)) \right) 
\\[10pt]
\displaystyle \qquad
=
\lim_{N\rightarrow\infty} \mathbb{E}^\mu\left( \exp(-r(N - 1, t))\cdot \mathbb{E}^\mu [\exp(-\langle \tw_{\frac{t}{N}} f, \widehat X_t\rangle ) \mid \mathcal{F}_{\frac{N-1}{N}t}]\right)
\\[10pt]
\displaystyle \qquad
=\lim_{N\rightarrow\infty} \mathbb{E}^\mu\left( \exp (-r(N - 1)) \cdot \mathbb{E}^{\widehat X_{\frac{N-1}{N}t}} \exp(-\langle \widetilde W_{\frac{t}{N}}f, \widehat X_{\frac{t}{N}}\rangle )\right) .
\end{array}
$$
We now apply property \eqref{4.3}  with  $s = \dfrac{t}{N}$, $\mu = \widehat X_{\frac{N -1}{N}}$, and $g = \tw_{\frac{t}{N}}f$, to  get
$$
\begin{array}{l}
\displaystyle
\mathbb{E}^\mu \exp (-\langle f, \widehat X_t\rangle - Y_t(\varphi)) = 
\lim_{N\rightarrow\infty}  \mathbb{E}^\mu \left( \exp (-r (N -1, t)) \exp(-\langle \tu_{\frac{t}{N}} \tw_{\frac{t}{N}} f, \widehat X_{\frac{N - 1}{N}t}\rangle )\right) 
\\[10pt]
\displaystyle \qquad 
=\lim_{N\rightarrow\infty}  \mathbb{E}^\mu \left( \exp (-r(N-2, t)) \exp(- \langle \dfrac{t}{N} \varphi + \tu_{\frac{T}{N}}\tw_{\frac{t}{N}} f, \widehat X_{\frac{N-1}{N}t}\rangle ) \right) 
\\[10pt]
\displaystyle \qquad =
\lim_{N\rightarrow\infty}  \mathbb{E}^\mu \left(\exp (-r (N - 2, t)) \cdot \mathbb{E}^\mu [\exp(-\langle \tw_{\frac{t}{N}}\tu_{\frac{t}{N}}\tw_{\frac{t}{N}} f, \widehat X_{\frac{N-1}{N}t}\rangle ) \mid \mathcal{F}_{\frac{N-2}{N}t}]\right).
\end{array}
$$
Using again the Markov property and \eqref{4.2} as before we obtain
$$
\begin{array}{l}
\displaystyle
\mathbb{E}^\mu\exp(-\langle f, \widehat X_t\rangle - Y_t(\varphi)) 
\\[10pt]
\displaystyle \qquad = 
\lim_{N\rightarrow\infty} \mathbb{E}^\mu\left( \exp(-r (N-2, t))\cdot \mathbb{E}^{\widehat X_{\frac{N-2}{N}t}} \exp (-\langle \tw_{\frac{t}{N}}C_{\frac{t}{N}}\tw_{\frac{t}{N}} f, \widehat X_{\frac{t}{N}}\rangle)\right)
\\[10pt]
\displaystyle \qquad  =
\lim_{N\rightarrow\infty} \mathbb{E}^\mu \left( \exp(-r(N-2, t)) \exp (-\langle (\tu_{\frac{t}{N}}\tw_{\frac{t}{N}})^2 f, \widehat X_{\frac{N - 2}{N}t}\rangle )\right),
\end{array}
$$
so, \eqref{a1} holds.

As in \cite{Iscoe86}, repeating the above procedure we have
$$
\mathbb{E}^\mu \exp(- \langle f, \widehat X_t\rangle - Y_t(\varphi)) = 
\lim_{N\rightarrow\infty} \exp(-\langle (\tu_{\frac{t}{N}}\tw_{\frac{t}{N}})^N f, \mu\rangle). 
$$  
We now use the following Trotter-Lee formula.
\begin{lemma}\label{lee}
	For any nonnegative $\varphi,f\in L^\infty(\rr^d)$ the following holds in $L^\infty(\rr^d)$:
	\begin{equation}
\label{4.5}
  \tv_t(\varphi,f)=\lim _{N\rightarrow\infty} (\tilde U_{t/N}\tilde W_{t/N})^N f,
 \quad \forall t>0.
\end{equation}
\end{lemma}
%
%
%

This gives us that 
\[
  \lim_{N\rightarrow\infty} \exp(-\langle (\tu_{\frac{t}{N}}\tw_{\frac{t}{N}})^N f, \mu\rangle). = \exp(-\langle 
\tv_t (\varphi, f), \mu \rangle).
\]
The proof of  Theorem \ref{th4.1} is  now finished.

We now prove Lemma \ref{lee}.
Let recall that $\tw_t f : = f + t \varphi$, where $\varphi \in b \mathcal{B}(\mathbb{R}^d)$, $t \geqslant  0$. Clearly $\tw_t f$ is the solution of the ordinary differential equation
$$
\left \{
\begin{array}{l}
\dfrac{d\tw}{dt} = \varphi, \\[10pt] W(0) = f. 
\end{array}
\right .
$$
Also $\tv_t(\varphi, f)$, $t \geqslant  0$ is the solution of  equation \eqref{4.4} and 
and $\tu_t f$ solution of \eqref{4.2}. 
Observe that if $\varphi \equiv 0$ then $\tv _t(0,f)=\tu _t(f)$.


Let us observe that $U_t(e^{-f})=\exp({-\tu _t(f)})$. Denoting by $V_t(\varphi,g)$ the solution of 
 \begin{equation}
\label{eq.v}
\left\{
\begin{array}{l}
\displaystyle  v _t= \Delta v  - v+ \sum_{k \geqslant  1} q_k v^k  -v \varphi, \qquad c\in \rr^d, t>0, \\[3mm]
v(\cdot, 0) = g,
\end{array}
\right.
  \end{equation} 
we have $V_t(\varphi,e^{-f})=\exp(-\tv_t(\varphi,f))$. 
We set $W_t(g)=ge^{-t\varphi}$. Thus $W_t(e^{-f})=\exp(-\tw _t(f))$.


The proof of Lemma \ref{lee} is done in two steps. In the first one we prove that is sufficient to show that 
 for any nonnegative $\varphi, \in L^\infty(\rr^d)$, we have 
\begin{equation}
\label{est.for.v}
  \|V_t(\varphi, e^{-f}) - (U_{t/n}W_{t/n})^n (e^{-f})\|_{L^\infty(\rr^d)}\rightarrow 0, \, n\rightarrow \infty.
\end{equation}
The second step consists in proving \eqref{est.for.v} by checking the conditions in \cite[Section 2]{chorin}. 
We start with the following estimate for the solutions of equation \eqref{eq.v}.

\begin{lemma}  
	\label{lemma.est.v}
	For any nonnegative $\varphi, f\in L^\infty(\rr^d)$, the solution of \eqref{eq.v} satisfies
\[
  e^{-t(1+\|\varphi\|_{L^\infty(\rr^d)})} e^{-\|f\|_{L^\infty(\rr^d)}}\leqslant  V_t(\varphi,e^{-f})\leqslant  1.
\]
\end{lemma}
\begin{proof}
	We use the maximum principle and the fact $v$ solution of  \eqref{est.for.v}  
satisfies 	$$\Delta v-v(1+\|\varphi\|_{L^\infty(\rr^d)})\leq v_t\leq   \Delta v.$$
 \end{proof}

With this result we can now prove  that \eqref{est.for.v} implies \eqref{4.5}. 
 Explicit computations shows  the following identity
\[
  (U_{t/n}W_{t/n})^n (e^{-f}) =\exp( - (\tu_{t/n}\tw_{t/n})^n(f)).
\]
Using that for any $a,b>0$ we have $|\log a-\log b|\leqslant  |a-b|/ (\min \{a,b\})$ we get
\begin{equation}
\label{equivalenta}
  \|\tv_t (\varphi,f)-(\tu_{t/n}\tw_{t/n})^n(f)\|_{L^\infty(\rr^d)}\leqslant  \frac{\|V_t(\varphi,e^{-f}) - (U_{t/n}W_{t/n})^n (e^{-f})\|_{L^\infty(\rr^d)}} {\min_{x\in \rr^d} \{\min \{  V_t(e^{-f}),  (U_{t/n}W_{t/n})^n (e^{-f})  \}  \}}. 
\end{equation}
Observe that for any nonnegative function $g$ we have 
\[
  W_t(g)\geqslant  e^{-t\|\varphi\|_{L^\infty(\rr^d)}} \min_{ \rr^d} \{g\}.
\]
Using Lemma \ref{lemma.est.v} with $\varphi \equiv 0$ we also obtain 
\[
  U_t(g)\geqslant  e^{-t}\min_{ \rr^d} \{g\}.
\]
The last two estimates and Lemma \ref{lemma.est.v} show that 
\[
  (U_{t/n}W_{t/n})^n (e^{-f}) \geq e^{-t(1+\|\varphi\|_{L^\infty(\rr^d)})} e^{-\|f\|_{L^\infty(\rr^d)}}.
\]
Then the denominator of the right hand side of \eqref{equivalenta} is bounded from below by some positive constant and it is sufficient to prove \eqref{est.for.v}.

We now prove that \eqref{est.for.v} holds. 
\begin{lemma} 
	Let us assume that $g\in L^\infty(\rr^d)$ with $0\leqslant  g\leqslant  1$ and the sequence $ (q_k)_{k\geqslant  1}$ satisfies
	$\sum _{k\geqslant 1} k^2q_k<\infty$. For any $t>0$ the following holds
	\begin{equation}
\label{convergence}
  \|V_t(\varphi,g) - (U_{t/n}W_{t/n})^n (g)\|_{L^\infty(\rr^d)}\rightarrow 0, \quad n\rightarrow \infty.
\end{equation}
\end{lemma}

\begin{proof}
	We first prove the result for $g\in W^{2,\infty}(\rr^d)$ with $0\leqslant  g\leqslant  1$ and then by $L^\infty(\rr^d)$-stability of the flows $V_t$ and  $(U_{t/n}W_{t/n})^n$ we obtain the desired result. 
	
	\textbf{Step I. Stability of the flow  	$V_t$.} We prove that for any $g_1,g_2\in L^\infty(\rr^d)$ with $0\leqslant  g_1,g_2\leqslant  1$ we have
	\begin{equation}
\label{stab.v}
  \|V_t(\varphi,g_1)-V_t(\varphi,g_2)\|_{L^\infty(\rr^d)} \leqslant  C(\|\varphi\|_{L^\infty(\rr^d)}, \sum_{k\geqslant 1} kq_k,T)\|g_1-g_2\|_{L^\infty(\rr^d)},\quad \forall t\in [0,T].
\end{equation}
Denoting by $v_1$ and $v_2$ the corresponding solutions we have $0\leqslant  v_1,v_2\leqslant  1$. Thus 
 $p=v_1-v_2$  satisfies
	\[
   p_t=\Delta p -p +p \sum_{k\geqslant 1}q_k(v_1^{k-1} + \dots+v_2^{k-1})-p\varphi.
\]
Using the sub-super solutions methods we find that 
$$\|p(t)\|_{L^\infty(\rr^d)}\leqslant  e^{tC(\|\varphi\|_{L^\infty(\rr^d)}, q)} \|g_1-g_2\|_{L^\infty(\rr^d)}.$$

\textbf{Step I. Local stability of the flow $(U_{t/n}W_{t/n})^n$ in $W^{2,\infty}(\rr^d)$. }
Let us consider $g\in W^{2,\infty}(\rr^d)$. 	
	The case of $L^\infty(\rr^d)$-norm easily follows since 
	\[
  \|W_t(g)\|_{L^\infty(\rr^d)}\leqslant  \|g\|_{L^\infty(\rr^d)},\quad 
  \|U_t(g)\|_{L^\infty(\rr^d)}\leqslant  \|g\|_{L^\infty(\rr^d)}.
\]

Let us now analyze the first derivative. Let us consider $u$ solution of \eqref{eq.v} with $\varphi=0$ and $p=u_{x_k}$. It follows that $p$ satisfies the equation
$$p_t=\Delta p-p-\sum _{k\geqslant  1}kq_k u^{k-1}p.$$
Choosing $\overline{p}=Me^{\alpha t}$ with $\alpha\geqslant  q-1$ and $M=\|g_{x_k}\|_{L^\infty(\rr^d)}$ we obtain that 
\[
  \|(U_t(g))_{x_k}\|_{L^\infty(\rr^d)}\leqslant  e^{\alpha t}\|g_{x_k}\|_{L^\infty(\rr^d)}.
\]
 Using that $W_t$ satisfies
 \[
  \|(W_t(g))_{x_k}\|_{L^\infty(\rr^d)}\leqslant  \|g_{x_k}\|_{L^\infty(\rr^d)} +t\|\varphi_{x_k}\|_{L^\infty(\rr^d)}\|g\|_{L^\infty(\rr^d)}
  \]
it follows that
\[
 \| (U_{t/n}W_{t/n})(g)\|_{\dot W^{1,\infty}(\rr^d)}\leqslant  e^{\alpha t /n}(\|g\|_{\dot W^{1,\infty}(\rr^d)} +\frac tn \|\varphi\|_{\dot W^{1,\infty}(\rr^d)}\|g\|_{L^\infty(\rr^d)}).
\]
Iterating the above argument we obtain that
\[
   \| (U_{t/n}W_{t/n})^n(g)\|_{\dot W^{1,\infty}(\rr^d)}\leqslant  e^{\alpha t }(\|g\|_{\dot W^{1,\infty}(\rr^d)} + t\|\varphi\|_{\dot W^{1,\infty}(\rr^d)} \|g\|_{L^\infty(\rr^d)}).
\]

Let us now denote $r=p_{x_i}$. Then $r$ satisfies
\[
  r_t=\Delta r- r +\sum _{k\geqslant  1} kq_k u^{k-1} r+\sum _{k\geqslant  1} k(k-1)q_ku^{k-2}p_{x_i}p_{x_k}.
\]
Choosing $\overline r=Me^{\beta t}$ and taking into account that 
\[|p_{x_i}p_{x_k}|\leqslant  e^{2\alpha t} 
(\|g\|_{\dot W^{1,\infty}(\rr^d)} + t\|\varphi\|_{\dot W^{1,\infty}(\rr^d)} \|g\|_{L^\infty(\rr^d)})^2=e^{2\alpha t C}
\]
 we can choose $$\beta(g) =q-1+(\sum_{k\geqslant  1} k^2q_k)(\|g\|_{\dot W^{1,\infty}(\rr^d)} + t\|\varphi\|_{\dot W^{1,\infty}(\rr^d)} )^2$$
 and $$M=\|q_{x_kx_i}\|_{L^\infty(\rr^d)}$$ to obtain that
 \[
  \|U_t(g)\|_{\dot W^{2,\infty}(\rr^d)}\leqslant  \|g\|_{\dot W^{2,\infty}(\rr^d)} \exp(t\beta(g)).
\]
Using the stability in $\dot W^{1,\infty}(\rr^d)$ and iterating the above estimate on $U_t$ we also obtain the local-stability in $\dot W^{2,\infty}(\rr^d)$.
  
  \textbf{Consistency}. We now prove that for any $g\in W^{2,\infty}$, $0\leqslant  g\leqslant  1$ we have
  \[
  \lim _{\epsilon\rightarrow 0} \frac{U_\epsilon W_\epsilon g- g}\epsilon - \Big(\Delta g - g+\sum _{k\geqslant  1}q_ku^k -g\varphi\Big) =0.
\]
We use that  $g\in W^{2,\infty}$, $0\leqslant  g\leqslant  1$ implies that our solution satisfies $U\in C^1([0,\infty], L^\infty(\rr^d))\cap C([0,T],W^{2,\infty}(\rr^d))$.  The explicit form of $W_\eps(g)= ge^{-\eps \varphi}$ give us the desired result.
\end{proof}


\begin{thebibliography}{BH}


\bibitem{Be11} 
L.  Beznea,   
Potential theoretical methods in the construction of measure-valued branching processes,
\textit{J. Europ.  Math. Soc.} \textbf{13} (2011), 685--707.

\bibitem{BeBo04}
L. Beznea, N. Boboc,
{\it Potential Theory and Right Processes},
Springer Series, Mathematics and Its Applications {\bf 572}, Kluwer, Dordrecht, 2004.

\bibitem{BeDeLu15}
L. Beznea, M. Deaconu,  and O. Lupa\c scu, 
{Branching processes for the  fragmentation equation}, 
{\it Stochastic Processes and their Applications}  {\bf 125}  (2015), 1861--1885.


\bibitem{BeDeLu16}
L. Beznea, M. Deaconu,  and O. Lupa\c scu, 
{Stochastic equation of fragmentation and branching processes related to avalanches}, 
{\it J. Stat. Physics} {\bf 162} (2016), 824--841.


\bibitem{BeLu14}
L. Beznea,   O. Lupa\c scu,
Measure-valued discrete branching Markov processes,
{\it Trans. Amer. Math. Soc.}   {\bf 368}  (2016),  5153--5176.

\bibitem{BeLuOp12}
L. Beznea,  O.  Lupa\c scu,  and A.-G.  Oprina, 
A unifying construction for measure-valued continuous and discrete branching processes.
In: {\it Complex Analysis and Potential Theory,} CRM Proceedings and Lecture Notes, 
vol. {\bf 55}, Amer. Math. Soc., Providence, RI, 2012,  pp. 47--59.

\bibitem{BeOp11}
L. Beznea, A.-G. Oprina,
{Nonlinear PDEs and measure-valued branching type processes},
{\it J. Math. Anal. Appl.} {\bf 384} (2011), 16--32.

\bibitem{BeOp14}
L. Beznea, A.-G. Oprina,
Bounded and $L^p$-weak solutions for nonlinear equations of measure-valued branching processes,
{\it Nonlinear Analysis} {\bf 107}  (2014), 34--46.

\bibitem{chorin}
A.J. Chorin,    T.J.R. Hughes,  M.F. McCracken,  and  J.E. Marsden,
Product formulas and numerical algorithms, 
{\it Communications on Pure and Applied Mathematics}  {\bf 31} (1978), 205--256.


\bibitem{escobedo}
M. Escobedo, E.  Zuazua,  
Large time behavior for convection-diffusion equations in $\rr^N$. 
\textit{Journal of Functional Analysis}c {\bf 100}  (1991), 119--161. 


\bibitem{INW68}
N. Ikeda, M. Nagasawa, and S. Watanabe,
 Branching Markov processes I,
{\it J. Math. Kyoto Univ.} {\bf 8} (1968), 233--278.

\bibitem{Iscoe86}
I. Iscoe,
A weighted occupation time for a class of measured-valued branching processes.
{\it Probab. Th. Rel. Fields} {\bf 71} (1986), 85--116.

\bibitem{Li11}
Z. H. Li, 
\textit{Measure-Valued Branching Markov Processes.}  Probab. Appl., Springer, 2011.

\bibitem{LuSt17}
O. Lupa\c scu, V. St\u anciulescu, 
Numerical solution for the non-linear Dirichlet problem of a branching process,
{\it Complex Analysis and Op. Th.} (2017),  DOI: 10.1007/s11785-017-0642-z, to appear.


\bibitem{Sh88}
M. Sharpe, 
\textit{General Theory of Markov Processes}. Academic Press, Boston, 1988.

\end{thebibliography}
\end{document}